\documentclass[11p]{article}
\usepackage{mathrsfs}
\usepackage[french]{babel}
\usepackage[latin1]{inputenc}

\usepackage{amssymb}
\usepackage{pdfsync}
\usepackage{array}
\usepackage{amsmath}
\usepackage{amsthm}
\usepackage{amsfonts}
\usepackage{amstext}
\usepackage{amsopn}
\usepackage{amsbsy}
\usepackage{shortvrb,fancybox}
\usepackage{rotating}
\usepackage{xy}
\xyoption{all}
\usepackage{multirow}

\newcommand{\Z}{\mathbb{Z}}
\newcommand{\Q}{\mathbb{Q}}
\newcommand{\aQ}{\overline{\mathbb{Q}}}

\newcommand{\C}{\mathbb{C}}

\newcommand{\Ax}{\mathbb{A}^{\times}}

\newcommand{\Kl}{K^{\lambda}}

\newcommand{\DK}{\Delta_{K}}
\newcommand{\OK}{\mathcal{O}_K}
\newcommand{\OL}{\mathcal{O}_L}

\newcommand{\q}{\mathfrak{q}}
\newcommand{\Pcq}{\mathcal{P}^{\mathfrak{q}}}
\newcommand{\OLq}{\mathcal{O}_{L^{\mathfrak{q}}}}

\newcommand{\gq}{\gamma_{\mathfrak{q}}}
\newcommand{\alq}{\alpha_{\mathfrak{q}}}
\newcommand{\alp}{\alpha_{\mathfrak{p}}}

\newcommand{\Kq}{K_{\mathfrak{q}}}

\newcommand{\eq}{e_{\mathfrak{q}}}

\newcommand{\sq}{\sigma_{\mathfrak{q}}}

\newcommand{\Lq}{L^{\mathfrak{q}}}
\newcommand{\bq}{\beta_{\mathfrak{q}}}

\newcommand{\bqb}{\overline{\beta_{\mathfrak{q}}}}

\newcommand{\Pq}{P_{\mathfrak{q}}}
\newcommand{\Tq}{T_{\mathfrak{q}}}

\newcommand{\phiEp}{\varphi_{E,p}}

\newcommand{\h}{\mathrm{H}}

\newcommand{\p}{\mathfrak{p}}
\newcommand{\po}{\mathfrak{p}_0}
\newcommand{\poq}{\mathfrak{p}^{\mathfrak{q}}_0}

\newcommand{\Poq}{\mathcal{P}^{\mathfrak{q}}_0}

\newcommand{\pL}{\mathfrak{p}_L}

\newcommand{\cN}{\mathcal{N}}

\newcommand{\Zp}{\mathbb{Z}_{p}}

\newcommand{\Qp}{\mathbb{Q}_{p}}
\newcommand{\Kp}{K_{\mathfrak{p}}}

\newcommand{\Fp}{\mathbb{F}_{p}}
\newcommand{\Fpx}{\mathbb{F}_{p}^{\times}}

\newcommand{\Ip}{I_{\mathfrak{p}}}

\newcommand{\ap}{a_{\p}}
\newcommand{\ato}{a_{\tau}}

\newcommand{\GL}{\mathrm{GL}}

\newcommand{\Gal}{\mathrm{Gal}}

\newcommand{\al}{\alpha}

\newcommand{\la}{\lambda}

\newcommand{\kip}{\chi_p}
\newcommand{\J}{\mathcal{J}}

\newcommand{\val}{\mathrm{val}}

\newcommand{\Mod}{\mkern-5mu \mod}

\def\Overline #1#2#3%
{\mkern #1mu
\overline {\mkern-#1mu #3 \mkern-#2mu}
\mkern #2mu }

\newtheorem*{TheoII}{Th\'eor\`eme II (Forme de la repr\'esentation semi-simplifi\'ee)}
\newtheorem*{TheoCheboEff}{Th\'eor\`eme (Chebotarev effectif,~\cite{[LMO]})}

\newtheorem*{TheoI}{Th\'eor\`eme I (Borne uniforme d'irr\'eductibilit\'e)}
\newtheorem*{TheoIII}{Th\'eor\`eme III}
\newtheorem*{Question}{Question}

\theoremstyle{plain}
\newtheorem{theoreme}{Th\'eor\`eme}[section]
\newtheorem{proposition}[theoreme]{Proposition}

\newtheorem{lemme}[theoreme]{Lemme}
\newtheorem{definition}[theoreme]{D\'efinition}
\theoremstyle{remark}

\newtheorem{remarques}[theoreme]{Remarques}
\newtheorem{notation}[theoreme]{Notation}
\newtheorem{notations}[theoreme]{Notations}

\AtBeginDocument{}

\title{Crit\`ere d'irr\'eductibilit\'e pour les courbes elliptiques semi-stables sur un corps de nombres}
\author{Agn\`es David \\
 {\small Laboratoire de math\'ematiques de Versailles } \\
{\small  Universit\'e de Versailles Saint-Quentin-en-Yvelines } \\
{\small 45 avenue des \'Etats-Unis }\\
{\small  78035 Versailles Cedex } \\
{\small  Agnes.David@ens-lyon.org} 
}
\date{} 

\begin{document}

\maketitle

\begin{abstract}
For a fixed number field and an elliptic curve defined and semi-stable over this number field,
 we consider the set of prime numbers~$p$ such that the Galois representation attached to the $p$-torsion points of the elliptic curve is reducible.
When the number field satisfies a certain necessary condition, we give an explicit bound, depending only on the number field and not on the semi-stable elliptic curve, for these primes.
This generalizes previous results of Kraus.
\end{abstract}

\section*{Introduction}	

Cet article traite de la r\'eductibilit\'e de la repr\'esentation galoisienne associ\'ee aux points de torsion d'une courbe elliptique semi-stable d\'efinie sur un corps de nombres.

Pour les points de torsion dont l'ordre est un nombre premier strictement sup\'erieur \`a une certaine borne, qui ne d\'epend que du corps de base et qu'on explicite,
on d\'emontre que, lorsque cette repr\'esentation est r\'eductible, sa semi-simplifi\'ee est isomorphe \`a la repr\'esentation associ\'ee \`a une courbe elliptique semi-stable et ayant des multiplications complexes sur le corps de base  (voir le th\'eor\`eme~II).

On en d\'eduit, lorsque le corps de base v\'erifie la condition n\'ecessaire que de telles courbes semi-stables et \`a multiplications complexes n'existent pas, une borne sup\'erieure uniforme pour les nombres premiers pouvant donner lieu \`a une repr\'esentation r\'eductible pour une courbe elliptique semi-stable (voir le th\'eor\`eme~I). On g\'en\'eralise ainsi des r\'esultats de Kraus (voir~\cite{[Krau2]} et~\cite{[Krau]}).
Enfin, on donne des exemples de corps satisfaisant le th\'eor\`eme~I (voir la partie~\ref{ssec:HypH}).

Les notations sont les suivantes.
On fixe un corps de nombres~$K$ et un plongement de~$K$ dans~$\C$ (dans tout le texte, on consid\'erera ainsi~$K$ comme un sous-corps de~$\C$). 
Soit~$\aQ$ la cl\^oture alg\'ebrique de~$\Q$ dans~$\C$~; on note~$G_K$ le groupe de Galois absolu~$\Gal(\aQ / K)$ de~$K$.
\`A partir de la partie~\ref{sec:glob} et pour toute la fin du texte, on suppose que l'extension~$K/\Q$ est galoisienne.

On fixe une courbe elliptique~$E$ d\'efinie sur $K$ qui est semi-stable sur~$K$, c'est-\`a-dire qui a, en toute place finie de~$K$, soit r\'eduction multiplicative (d\'eploy\'ee ou non) soit bonne r\'eduction. 

On fixe un nombre premier~$p$ sup\'erieur ou \'egal \`a~$5$ et non ramifi\'e dans~$K$.
 On note~$E_p$ l'ensemble des points de~$E$ dans~$\aQ$ qui sont de~$p$-torsion~; 
  c'est un espace vectoriel de dimension~$2$ sur~$\Fp$, sur lequel le groupe de Galois absolu~$G_K$ de~$K$ agit~$\Fp$-lin\'eairement.
On d\'esigne par~$\phiEp$ la repr\'esentation de~$G_K$ ainsi obtenue~; 
elle prend ses valeurs dans le groupe~$\GL(E_p)$ qui, apr\`es choix d'une base pour~$E_p$, est isomorphe \`a~$\GL_2(\Fp)$.
On note enfin~$\kip$ le caract\`ere cyclotomique de~$G_K$ dans~$\Fpx$~; il co\"incide avec le d\'eterminant de la repr\'esentation~$\phiEp$.

Lorsque le corps de base~$K$ est celui des rationnels~$\Q$,
 Mazur a d\'emontr\'e dans~\cite{[Maz]} que, pour~$p$ strictement sup\'erieur \`a~$7$ et~$E$ semi-stable
 (ou pour~$p$ strictement sup\'erieur \`a~$163$ et~$E$ d\'efinie sur~$\Q$, pas forc\'ement semi-stable),
 la repr\'esentation~$\phiEp$ est irr\'eductible.
 Le th\'eor\`eme de Mazur constitue, pour le corps~$\Q$, une partie de la r\'eponse au \og probl\`eme de Serre uniforme \fg\ (voir la partie~4.3 de~\cite{[Se]})
   qui consiste \`a d\'eterminer s'il existe une borne~$B_K$ v\'erifiant~:
   pour toute courbe elliptique~$E$ d\'efinie sur~$K$ et sans multiplication complexe sur~$\aQ$ et tout nombre premier~$p$ strictement sup\'erieur \`a~$B_K$,
la repr\'esentation~$\phiEp$ est surjective.
On pourra \`a ce sujet consulter les travaux r\'ecents de Bilu, Parent et Rebolledo (voir~\cite{[BPR]}).

 On supposera donc dans tout ce texte que le corps~$K$ est diff\'erent de~$\Q$ et on se limitera \`a l'\'etude des courbes elliptiques semi-stables sur~$K$.
 Dans la direction d'une g\'en\'eralisation du r\'esultat de Mazur pour les courbes elliptiques semi-stables, Kraus a pos\'e, et trait\'e dans certains cas, la question suivante (voir \cite{[Krau2]} et~\cite{[Krau]}).

\begin{Question}
Existe-t-il une borne~$C_K$, ne d\'ependant que du corps de nombres~$K$ et v\'erifiant~:
pour toute courbe elliptique~$E$ d\'efinie et semi-stable sur~$K$ et tout nombre premier~$p$ strictement sup\'erieur \`a~$C_K$,
la repr\'esentation~$\phiEp$ est irr\'eductible ?
\end{Question}

Kraus a r\'esolu cette question, en explicitant la borne~$C_K$, pour les corps quadratiques dans~\cite{[Krau2]} et pour les corps v\'erifiant une certaine condition~$(C)$ dans~\cite{[Krau]}.
L'appendice~B de~\cite{[Krau]} mentionne également que le th\'eor\`eme non effectif de~\cite{[Mom]}, alli\'e aux bornes de Merel (voir~\cite{[Mer96]}) sur l'ordre des points de torsion d'une courbe elliptique, donne des r\'esultats semblables pour les corps~$K$ ne contenant le corps de classes de Hilbert d'aucun corps quadratique imaginaire.

Une condition n\'ecessaire sur le corps~$K$ pour que la r\'eponse  \`a la question ci-dessus soit positive est qu'il n'existe pas de courbe elliptique d\'efinie et semi-stable sur~$K$, ayant de plus des multiplications complexes sur~$K$ (on remarque qu'une telle courbe elliptique a alors bonne r\'eduction en toute place de~$K$).
En effet, pour une telle courbe elliptique~$E$, ayant des multiplications complexes par un corps quadratique imaginaire~$L$ contenu dans~$K$, la repr\'esentation~$\phiEp$ est r\'eductible pour tout nombre premier~$p$ d\'ecompos\'e dans~$L$ (voir proposition~1.2 de~\cite{[Bil]}).

On montre ici que cette condition est suffisante.
 Lorsqu'elle est r\'ealis\'ee, on donne de plus une forme explicite pour une borne~$C_K$ r\'epondant \`a la question (voir d\'efinition~\ref{def:CK}).
 
 \begin{TheoI}
 On suppose que le corps~$K$ est galoisien sur~$\Q$.
 Les assertions suivantes sont \'equivalentes.
\begin{enumerate}
\item Il n'existe pas de courbe elliptique d\'efinie sur~$K$, ayant des multiplications complexes sur~$K$ et partout bonne réduction sur~$K$.
\item Il existe une borne~$C_K$, ne d\'ependant que du corps de nombres~$K$ et v\'erifiant~:
pour toute courbe elliptique~$E$ d\'efinie et semi-stable sur~$K$ et tout nombre premier~$p$ strictement sup\'erieur \`a~$C_K$,
la repr\'esentation~$\phiEp$ est irr\'eductible.
\end{enumerate}
Lorsqu'elles sont v\'erifi\'ees, la borne~$C_K$ donn\'ee par la d\'efinition~\ref{def:CK} satisfait la deuxi\`eme assertion.
\end{TheoI}

On trouvera dans la partie~\ref{ssec:HypH} (th\'eor\`eme~III) des exemples de corps~$K$ satisfaisant le th\'eor\`eme~I.

On renvoie aux d\'efinitions \ref{def:C1(K)}, \ref{def:C2(K)}, \ref{def:C(K,n)}, \ref{def:CK}, pour une formule explicite pour la borne~$C_K$.
Elle d\'epend du degr\'e, du discriminant, du nombre de classes et du r\'egulateur du corps~$K$,
 ainsi que d'une constante absolue intervenant dans une forme explicite du th\'eor\`eme de Chebotarev due \`a Lagarias, Montgomery et Odlyzko (voir la partie~\ref{ssec:Cheboeff} et~\cite{[LMO]}).
 La d\'etermination de la borne~$C_K$ utilise \'egalement les bornes de Bugeaud et Gy{\H{o}}ry (voir~\cite{[BuGy]}) sur la hauteur d'un g\'en\'erateur d'un id\'eal (voir partie~\ref{ssec:BuGy})
 et, de mani\`ere cruciale, la borne uniforme sur l'ordre des points de torsion d'une courbe elliptique de Merel et Oesterl\'e qui figure dans les introductions de~\cite{[Mer96]} et~\cite{[Pa]} (voir la proposition~\ref{prop:mmredJK}).

Le th\'eor\`eme~I d\'ecoule du th\'eor\`eme~II suivant, qui pr\'ecise, pour les courbes elliptiques semi-stables, les r\'esultats de~\cite{[Dav08]} (th\'eor\`eme~I),~\cite{[Cici]} (th\'eor\`eme~II) et (dans une forme non explicite)~\cite{[Mom]} (th\'eor\`eme~A).

\begin{TheoII}
On suppose que le corps~$K$ est galoisien sur~$\Q$ et qu'il existe une courbe elliptique~$E$ semi-stable sur~$K$ et un nombre premier~$p$ strictement supérieur à~$C_K$ tels que la représentation~$\phiEp$ est réductible.

Alors il existe une courbe elliptique~$E'$ définie sur~$K$, ayant des multiplications complexes sur~$K$ et partout bonne réduction sur~$K$,
 telle que que la semi-simplifi\'ee de la repr\'esentation~$\phiEp$ est isomorphe \`a la repr\'esentation~$\varphi_{E',p}$.
\end{TheoII}

On remarque que le type supersingulier du th\'eor\`eme~II de~\cite{[Cici]} n'appara\^it pas ici~; en effet, il ne se produit que pour des courbes elliptiques ayant r\'eduction additive aux places de~$K$ au-dessus de~$p$.

Dans tout le texte, on suppose que la courbe elliptique~$E$ est semi-stable sur~$K$  et que la repr\'esentation~$\phiEp$ est r\'eductible.
La courbe~$E$ poss\`ede alors un sous-groupe d'ordre~$p$ d\'efini sur~$K$.
On fixe un tel sous-groupe~$W$~; il lui est associ\'e une isog\'enie de~$E$ de degr\'e~$p$, d\'efinie sur~$K$.
L'action de~$G_K$ sur~$W(\aQ)$ est donn\'ee par un caract\`ere continu de~$G_K$ dans $\Fpx$~; on le note $\la$ et, suivant la terminologie introduite dans~\cite{[Maz]}, on l'appelle le caract\`ere d'isog\'enie associ\'e au couple~$(E,W)$.
On fixe \'egalement une base de $E_p$ dont le premier vecteur engendre~$W(\aQ)$~; dans cette base la matrice de la repr\'esentation~$\phiEp$ est triangulaire sup\'erieure, de caract\`eres diagonaux~$(\la, \kip\la^{-1})$.

La m\'ethode de d\'emonstration du th\'eor\`eme~I suit celle de~\cite{[Dav08]} et~\cite{[Cici]}, inspir\'ee de~\cite{[Mom]}.
Plut\^ot qu'appliquer les r\'esultats de~\cite{[Cici]} (ou~\cite{[Dav08]}) \`a une courbe elliptique semi-stable,
 on a choisi d'en retracer la d\'emonstration, qui donne une meilleure borne~$C_K$ (comparer les d\'efinitions~\ref{def:C2(K)}, \ref{def:C(K,n)} et~\ref{def:CK} du pr\'esent texte avec les d\'efinitions~2.10,~2.13 et~4.3 de~\cite{[Cici]})
 et permet d'exprimer directement la semi-simplifi\'ee de la repr\'esentation~$\phiEp$ au lieu de sa puissance douzi\`eme. 
N\'eanmoins, lorsque c'\'etait possible, on s'est r\'ef\'er\'e \`a des r\'esultats connus dans le cas g\'en\'eral, en pr\'ecisant l'am\'elioration propre aux courbes semi-stables.

La premi\`ere partie de l'article rappelle les propri\'et\'es locales du caract\`ere d'isog\'enie,
en l'occurence ses restrictions aux sous-groupes de d\'ecomposition des places finies de~$K$ hors de~$p$ et ses restrictions aux sous-groupes d'inertie des places de~$K$ au-dessus de~$p$.
 
 La deuxi\`eme partie regroupe des propri\'et\'es globales du caract\`ere d'isog\'enie et du corps~$K$ (dont on suppose \`a partir de ce point du texte qu'il est galoisien sur~$\Q$).
 On y trouve notamment la th\'eorie du corps de classes globale pour le carr\'e du caract\`ere d'isog\'enie, qui relie  entre elles les propri\'et\'es locales de la premi\`ere partie.
 
 La troisi\`eme partie comporte la d\'emonstration du th\'eor\`eme~II (partie~\ref{ssec:thm2}), dont le th\'eor\`eme~I est une cons\'equence directe.
  Enfin, on donne dans la partie~\ref{ssec:HypH} (th\'eor\`eme~III) une liste de corps satisfaisant les assertions du th\'eor\`eme~I.

\section{Propri\'et\'es locales du caract\`ere d'isog\'enie} \label{sec:proploc}

\subsection{Aux places finies hors de~$p$}

Soient~$q$ un nombre premier rationnel diff\'erent de~$p$ et~$\q$ un id\'eal premier de~$K$ au-dessus de~$q$.

\begin{proposition} \label{prop:inerhorsp}
Le caract\`ere~$\la$ est non ramifi\'e en~$\q$.
\end{proposition}

\begin{proof}
Lorsque la courbe~$E$ a bonne r\'eduction en~$\q$, le r\'esultat est donn\'e par le th\'eor\`eme~1 de~\cite{[SeTa]}.

Lorsque la courbe~$E$ a r\'eduction multiplicative en~$\q$, il découle de la partie~\S1.12 de~\cite{[Se]}.
Pour un énoncé précis, on renvoie à la proposition~1.1.1 de~\cite{[Dav08]}, la proposition~3 de~\cite{[Krau]} ou la proposition~1.4 de~\cite{[Cici]} (avec l'entier~$\eq$ \'egal \`a~$1$, voir~\S1.1.1 de~\textit{loc. cit.}).
\end{proof}  

\begin{notations}\label{not:q}\ 
\begin{enumerate}
\item On fixe dans~$G_K$ un \'el\'ement de Frobenius~$\sq$ pour~$\q$~; comme le caract\`ere~$\la$ est d'image ab\'elienne et non ramifi\'e en~$\q$, l'image de~$\sq$ par~$\la$ est ind\'ependante de ce choix.
\item On note~$N\q$ le cardinal du corps r\'esiduel de~$K$ en~$\q$.
\end{enumerate}
\end{notations}

\begin{proposition} \label{prop:redmultq} 
On suppose que~$E$ a r\'eduction multiplicative en~$\q$.
Alors~$\la^2(\sq)$ vaut~$1$ ou~$(N\q)^2$ modulo~$p$.
\end{proposition}

\begin{proof}
Voir par exemple la proposition~1.1.1 de~\cite{[Dav08]}  ou la proposition~1.4 de~\cite{[Cici]}.
\end{proof}

On rappelle \'egalement les r\'esultats de~\cite{[SeTa]} qui serviront dans la suite (voir aussi la proposition~1.1.2 de~\cite{[Dav08]} ou la partie~1.3.2 de~\cite{[Cici]}).
\begin{proposition}[Th\'eor\`eme~3 de~\cite{[SeTa]}] \label{prop:frobhorspbonST}
On suppose que~$E$ a bonne r\'eduction en~$\q$. Alors le polyn\^ome caract\'eristique de l'action de~$\sq$ sur le module de Tate en~$p$ de~$E$ est \`a coefficients dans~$\Z$ (et ind\'ependant de~$p$)~; ses racines ont pour valeur absolue complexe~$\sqrt{ N\q}$.
\end{proposition}

\begin{notations}\label{not:Lq}\ 
\begin{enumerate}
\item On note~$\Pq(X)$ le polyn\^ome caract\'eristique de l'action de~$\sq$ sur le module de Tate en~$p$ de~$E$.
Il est de la forme~$X^2 - \Tq X + N\q$ avec $\Tq$ un entier de valeur absolue inf\'erieure ou \'egale \`a~$2\sqrt{N\q}$.
Son discriminant~$\Tq^2 - 4N\q$ est donc un entier n\'egatif et ses racines sont conjugu\'ees complexes l'une de l'autre.
\item On note~$\Lq$ le sous-corps engendr\'e dans~$\C$ par les racines de~$\Pq(X)$~; le corps~$\Lq$ est soit~$\Q$ soit un corps quadratique imaginaire.
\end{enumerate}
\end{notations}

\begin{proposition} \label{prop:frobhorspbon}
On suppose que~$E$ a bonne r\'eduction en~$\q$. 
Soit~$\Pcq$ un id\'eal premier de~$\Lq$ au-dessus de~$p$.
Alors les images dans~$\OLq / \Pcq$ des racines de~$\Pq(X)$ sont dans~$\Fpx$~;
 il existe une racine~$\bq$ de~$\Pq(X)$ v\'erifiant~:
 $$
 \left( \la(\sq) , \left( \kip\la^{-1} \right) (\sq) \right) = \left( \bq \Mod \Pcq , \bqb \Mod \Pcq \right).
 $$
\end{proposition}

\begin{proof}
Cela résulte de la réductibilité de la représentation~$\phiEp$
(voir aussi la proposition~1.1.2 de~\cite{[Dav08]} ou la proposition~1.8 de~\cite{[Cici]}).
\end{proof}

\subsection{Aux places finies au-dessus de~$p$~: action de l'inertie}

Soit~$\p$ un id\'eal premier de~$K$ situ\'e au-dessus de~$p$.
On fixe dans~$G_K$ un sous-groupe d'inertie~$\Ip$ pour~$\p$.

 \begin{proposition} \ 
 \begin{enumerate}
 \item On suppose que~$E$ a r\'eduction multiplicative en~$\p$. Alors $\la$ restreint \`a~$\Ip$ est trivial ou \'egal \`a~$\kip$.
\item On suppose que~$E$ a bonne r\'eduction en~$\p$. Alors cette r\'eduction est ordinaire et $\la$ restreint \`a~$\Ip$ est trivial ou \'egal \`a~$\kip$.
\end{enumerate}
\end{proposition}

\begin{proof}
Lorsque~$E$ a r\'eduction multiplicative en~$\p$, le r\'esultat est donn\'e par la proposition~13 (et son corollaire) de~\cite{[Se]},~\S1.12.

 Lorsque~$E$ a bonne r\'eduction en~$\p$, on suppose d'abord par l'absurde que~$E$ a r\'eduction supersinguli\`ere en~$\p$.
 Alors, par la proposition~12 (\S1.11(2)) de~\cite{[Se]}, l'image par~$\Ip$ de~$\phiEp$ est un groupe cyclique d'ordre~$p^2 - 1$. Ceci contredit l'hypoth\`ese que la repr\'esentation~$\phiEp$ est r\'eductible et que l'ordre de son image divise donc~$p(p-1)^2$.
 La courbe~$E$ a ainsi réduction ordinaire en~$\p$~; le résultat découle alors de la  proposition~11 et son corollaire dans~\cite{[Se]} (\S1.11(1)) (voir aussi la proposition~3 de~\cite{[Krau]}).
\end{proof}

\begin{notation} \label{not:ap}
On note~$\ap$ l'entier, \'egal \`a~$0$ ou~$1$, tel que~$\la$ restreint au sous-groupe d'inertie~$\Ip$ co\"incide avec~$\kip^{\ap}$.
\end{notation}

\section{Des propri\'et\'es globales}\label{sec:glob}

\subsection{Th\'eorie du corps de classes pour le caract\`ere~$\la^2$}\label{cdc}  

La th\'eorie du corps de classes globale associe au caract\`ere ab\'elien~$\la^2$ de~$G_K$ dans~$\Fpx$ un morphisme de groupes du groupe des id\`eles~$\Ax_K$ de~$K$ dans~$\Fpx$ qui est continu et trivial sur les id\`eles diagonales.
On note~$r$ ce morphisme.

Pour toute place (finie ou infinie)~$\nu$ de~$K$, on note~$K_{\nu}$ le compl\'et\'e de~$K$ en~$\nu$ et~$r_{\nu}$ la compos\'ee de l'injection de~$K_{\nu}^{\times}$ dans les id\`eles~$\Ax_{K}$ et du morphisme~$r$.
Lorsque~$\nu$ est une place finie, on note~$U_{K_{\nu}}$ les unit\'es du corps local~$K_{\nu}$~; dans ce cas, on utilise indiff\'eremment en indice la place~$\nu$ et l'id\'eal maximal de~$K$ qui lui correspond.

Les propri\'et\'es locales du caract\`ere~$\la^2$ \'etablies dans la partie~\ref{sec:proploc} ont les interpr\'etations suivantes pour le morphisme~$r$
(voir \'egalement~\cite{[Dav08]} \S2.1,~\cite{[Cici]} \S2.1 ou~\cite{[Bil]} \S2.3) pour des discussions semblables)~:
\begin{enumerate}
\item[(i)] pour une place infinie~$\nu$ de~$K$, l'application~$r_{\nu}$ est triviale~;
\item[(ii)] pour un id\'eal maximal~$\q$ de~$K$ qui n'est pas au-dessus de~$p$, l'application~$r_{\q}$ envoie toute uniformisante de~$\Kq$ sur~$\la^2(\sq)$ et sa restriction \`a~$U_{\Kq}$ est triviale~;
\item[(iii)] pour un id\'eal maximal~$\p$ de~$K$ au-dessus de~$p$, la restriction de~$r_{\p}$ \`a~$U_{\Kp}$ co\"incide avec la compos\'ee d'applications
$$
U_{\Kp} \xrightarrow{N_{\Kp/\Qp}} U_{\Qp} \xrightarrow[\text{modulo } p]{\text{r\'eduction}} \Fpx \xrightarrow[\text{\`a la puissance } - 2 \ap]{\text{\'el\'evation}} \Fpx.
$$
\end{enumerate}

\begin{notations} Pour toute place~$\nu$ de~$K$, on note~$\iota_{\nu}$ le plongement de~$K$ dans le compl\'et\'e~$K_{\nu}$. Pour tout id\'eal maximal~$\mathfrak{L}$ de~$K$, on note~$\mathrm{val}_{\mathfrak{L}}$ la valuation de~$K$ associ\'ee \`a~$\mathfrak{L}$ dont l'image est~$\Z$.
\end{notations}

\begin{proposition}\label{prop:recla2}
Soit~$\al$ un \'el\'ement de~$K$ non nul et premier \`a~$p$~;
 alors on a~:
 $$
 \prod\limits_{\q \nmid p} \la^2 \left(\sq\right)^{\mathrm{val}_{\q}(\al)} =
 \prod_{\p | p} N_{\Kp/\Qp}\left(\iota_{\p}(\al)\right)^{2\ap} \Mod p.
 $$
\end{proposition}

\begin{proof}
La d\'emonstration repose sur le m\^eme principe que celle du lemme~1 de~\cite{[Mom]} (voir aussi la proposition~ 2.2.1 de~\cite{[Dav08]} ou la proposition~2.4 de~\cite{[Cici]}).
L'image par~$r$ de l'id\`ele principale~$(\iota_{\nu}(\al))_{\nu}$ \'etant triviale, on a 
 (tous les produits \'etant finis)~:
$$
\begin{array}{r c l}
 1 & = &  r \left(\left(\iota_{\nu}(\al)\right)_{\nu}\right) 
    =  \prod\limits_{\nu}  r_{\nu}\left(\iota_{\nu}(\al)\right) 
   =  \prod\limits_{\nu | \infty} 1 \times \prod\limits_{\q \nmid p}  r_{\q}\left(\iota_{\q}(\al)\right) \times \prod\limits_{\p | p} r_{\p}\left(\iota_{\p}(\al)\right) \\
   & = & \prod\limits_{\q \nmid p} \la^2 (\sq )^{\mathrm{val}_{\q}(\al)} \times \prod\limits_{\p | p}  \left( N_{K_{\p}/\Q_{\p}}\left(\iota_{\p}(\alpha)\right)^{ - 2\ap} \Mod  p \right) .\\
\end{array}
$$
\end{proof}

\begin{notations}\label{not:po}\ 
Pour toute la suite du texte~:
\begin{itemize}
\item[$\bullet$] on suppose que l'extension~$K/\Q$ est galoisienne et on note~$G$ son groupe de~Galois~;
\item[$\bullet$] on fixe un id\'eal~$\po$ de~$K$ au-dessus de~$p$~;
\item[$\bullet$] pour tout \'el\'ement~$\tau$ de~$G$, on note~$\ato$ l'entier~$\ap$ associ\'e \`a l'id\'eal~$\p = \tau^{-1}(\po)$~;
\item[$\bullet$] on note~$\cN$ l'application de~$K$ dans lui-m\^eme qui envoie un \'el\'ement~$\alpha$ sur la norme tordue par les entiers~$(2\ato)_{\tau \in G}$
(c'est-\`a-dire qu'on a~$\cN (\al) = \prod_{\tau \in G} \tau(\alpha)^{2\ato}$)~; 
on remarque que l'application~$\cN$ pr\'eserve~$K^{\times}$,~$\OK$ et les \'el\'ements de~$K$ premiers \`a~$p$. 
\end{itemize}
\end{notations}

Avec ces notations, la proposition~\ref{prop:recla2} admet la reformulation globale suivante. 

\begin{proposition} \label{prop:recmuglob}
 Soit~$\al$ un \'el\'ement de~$K$ non nul et premier \`a~$p$~;
 alors on a~:
 $$
 \prod\limits_{\q \nmid p} \la^2 \left(\sq\right)^{\mathrm{val}_{\q}(\al)} = \iota_{\po} \left( \cN (\al) \right) \Mod \po.
 $$
\end{proposition}

\begin{proof}
L'\'el\'ement~$\prod_{\p | p} N_{\Kp/\Qp}\left(\iota_{\p}(\al)\right)^{2\ap}$  de~$\Zp$ est l'image par l'injection canonique~$\iota_{\po}$ de~$K$ dans son compl\'et\'e~$K_{\po}$ de l'\'el\'ement~$\prod_{\tau \in G} \tau(\al)^{2\ato}$ de~$K$. 
\end{proof}

\subsection{Borne pour la hauteur d'un g\'en\'erateur d'id\'eal}\label{ssec:BuGy}

Cette partie reprend les r\'esultats de la partie~2.3 de~\cite{[Dav08]} (voir aussi les parties~2.3 et~2.4 de~\cite{[Cici]}) qui donnent une borne pour toutes les valeurs absolues archim\'ediennes de la norme tordue~$\cN$ d'un certain g\'en\'erateur d'id\'eal dans~$\OK$.

\begin{notations}
On note~$d_K$ le degré du corps~$K$ sur~$\Q$ et~$h_K$ son nombre de classes d'idéaux.
Soit~$\al$ un \'el\'ement non nul de~$K$. On note~$\h$ la hauteur absolue de~$\al$ d\'efinie par
$$
\h(\al) =  \left( \prod  \limits_{\nu \text{ place de } K}  \max \left(1, \left| \al \right|_{\nu} \right)\right)^{1/d_K},
$$
avec les normalisations suivantes pour les valeurs absolues~$| \cdot |_{\nu}$~:
\begin{itemize}
\item[$\bullet$] si~$\nu$ est une place r\'eelle, associ\'ee \`a un \'el\'ement~$\tau$ de~$G$,~$| \cdot |_{\nu} = | \tau(\cdot) |_{\C}$~;
\item[$\bullet$] si~$\nu$ est une place complexe, associ\'ee \`a un \'el\'ement~$\tau$ de~$G$,~$| \cdot |_{\nu} = | \tau(\cdot) |^2_{\C}$~;
\item[$\bullet$] si~$\nu$ est une place finie, associ\'ee \`a un id\'eal premier~$\mathfrak{L}$ de $K$,~$| \cdot |_{\nu} = (N\mathfrak{L})^{ -  \mathrm{val}_{\mathfrak{L}}(\cdot)}$.
\end{itemize}
Lorsque~$\al$ est un entier de~$K$, les seules places apportant une contribution non triviale dans le produit d\'efinissant~$ \h ( \al ) $ sont les places infinies.
\end{notations}

\begin{lemme}\label{lem:hautalpha}
Soit~$\al$ un \'el\'ement non nul de~$\OK$~; alors, pour tout~$\tau$ dans~$G$, on~a~:
$$
\left| \tau \left( \cN(\al) \right) \right|_{\C} \leq \h(\al) ^{2d_K}.
$$
\end{lemme}

\begin{proof}
La d\'emonstration repose sur un calcul semblable \`a celui de la démonstration de la proposition~2.3.1 de~\cite{[Dav08]}
(voir aussi le lemme~2 (partie~2.3) de~\cite{[Cici]}), 
\`a ceci pr\`es que les entiers intervenant en exposant dans la d\'efinition de~$\cN$ valent ici~$0$ ou~$2$, alors qu'ils sont compris entre~$0$ et~$12$ dans \textit{loc. cit.}.
\end{proof}

\begin{notations}\label{not:C1} Suivant~\cite{[BuGy]}, on note~:
\begin{itemize}
\item[$\bullet$] $R_K$ le r\'egulateur de $K$~;
\item[$\bullet$] $r_K$ le rang du groupe des unit\'es de $K$ ($K$ \'etant suppos\'e galoisien sur~$\Q$,~$r_K$ vaut $d_K - 1$ si $K$ est totalement r\'eel et~$\frac{d_K}{2}-1$ sinon)~;
\item[$\bullet$] $\delta_K$ un r\'eel strictement positif minorant~$d_K \ln \left( \h(\al) \right)$ pour  tout \'el\'ement non nul~$\al$ de~$K$ qui n'est pas une racine de l'unit\'e~; 
si~$d_K$ est \'egal \`a~$2$, on peut prendre~$\delta_K$ \'egal \`a~$\frac{\ln 2}{r_K +1}$ ; si~$d_K$ est sup\'erieur ou \'egal \`a $3$, on peut prendre~$\delta_K$ \'egal \`a~$\frac{1}{53d_K\ln(6d_K)}$ ou  $\frac{1}{1201}\left(\frac{\ln \left( \ln d_K  \right) }{\ln d_K}\right)^3$.
\end{itemize}
\end{notations}

\begin{definition}[Borne~$C_1(K) $]\label{def:C1(K)}
On pose~:
$$
C_1(K) = \frac{r_K^{r_K +1} \delta_K^{-(r_K -1)}}{2}.
$$
\end{definition}
On remarque que~$C_1(K)$ peut s'exprimer en n'utilisant que le degr\'e~$d_K$ de~$K$ sur~$\Q$. Avec ces notations, le lemme~$2$ (partie~$3$) de~\cite{[BuGy]} s'\'ecrit de la mani\`ere suivante.
 
\begin{lemme}\label{lem:Yann}
Pour tout \'el\'ement non nul~$\al$ de~$\OK$, il existe une unit\'e~$u$ de~$K$ v\'erifiant~:
$$
\h(u\al) \leq \left| N_{K/\Q}(\al) \right|^{1/d_K} \exp\left( C_1(K)R_K \right).
$$
\end{lemme}

\begin{definition}[Borne~$C_2(K)$]\label{def:C2(K)}
On pose~:
$$
C_2(K) = \exp \left( 2 d_K C_1(K) R_K \right).
$$
\end{definition}
On remarque que~$C_2(K)$ ne d\'epend que du degr\'e et du r\'egulateur de~$K$.

\begin{proposition}\label{prop:hautgen}
 Soit~$\mathfrak{L}$ un id\'eal maximal de~$K$.
Il existe un g\'en\'erateur~$\gamma_{\mathfrak{L}}$ de~$\mathfrak{L}^{h_K}$ satisfaisant pour tout~$\tau$ dans~$G$~:
$$
\left| \tau \left( \cN(\gamma_{\mathfrak{L}}) \right) \right|_{\C} \leq \left(N \mathfrak{L} \right)^{2h_K} C_2(K).
$$
\end{proposition}

\begin{proof}
Par le lemme~\ref{lem:Yann}, on peut choisir un g\'en\'erateur de l'id\'eal principal~$\mathfrak{L}^{h_K}$ dont la hauteur est born\'ee par $(N \mathfrak{L})^{h_K / d_K} \exp\left( C_1(K)R_K \right)$~;
on conclut alors par le lemme~\ref{lem:hautalpha}
(voir aussi la proposition~2.3.4 de~\cite{[Dav08]} ou la proposition~2.11 de~\cite{[Cici]}).
\end{proof}

\subsection{Une version effective du th\'eor\`eme de Chebotarev}\label{ssec:Cheboeff}

\begin{TheoCheboEff}
Il existe une constante absolue et effectivement calculable~$A$ ayant la propri\'et\'e suivante~:
soient~$M$ un corps de nombres,~$N$ une extension finie galoisienne de~$M$,~$\Delta_N$ le discriminant de~$N$,~$C$ une classe de conjugaison du groupe de Galois~$\Gal(N/M)$~;
alors il existe un id\'eal premier de~$M$, non ramifi\'e dans~$N$, dont la classe de conjugaison des 
frobenius dans l'extension~$N/M$ est la classe de conjugaison~$C$ et dont la norme dans l'extension~$M/\Q$ est un nombre premier rationnel inf\'erieur ou \'egal \`a~$2(\Delta_N)^A$. 
\end{TheoCheboEff}

\begin{notation}
On note~$\Delta_K$ le discriminant de~$K$.
\end{notation}

\begin{definition}[Ensemble d'id\'eaux~$\J_K$]\label{def:JK}
On note~$\J_K$ l'ensemble des id\'eaux maximaux de~$K$ dont la norme dans l'extension~$K/\Q$ est un nombre premier rationnel totalement d\'ecompos\'e dans~$K$ et inf\'erieur ou \'egal \`a~$2(\Delta_{K})^{Ah_K}$.
\end{definition}

\begin{proposition}\label{prop:genclid}
Toute classe d'id\'eaux de~$K$ contient un id\'eal de~$\J_K$.
\end{proposition}

\begin{proof}
Voir~\cite{[Dav08]}, proposition~3.1.2 ou~\cite{[Cici]}, proposition~4.2.
\end{proof}

\section{D\'emonstrations des th\'eor\`emes}

\subsection{Compatibilit\'e des actions au-dessus et hors de~$p$}

\begin{notations}\label{not:poq}
 Soit~$\q$ un id\'eal maximal de~$K$ premier \`a~$p$ en lequel~$E$ a bonne r\'eduction.
\begin{enumerate}
\item On fixe un id\'eal~$\poq$ de~$K\Lq$ au-dessus de~$\po$ (voir notations~\ref{not:Lq} et~\ref{not:po}).
Comme le corps~$\Lq$ est soit~$\Q$, soit un corps quadratique imaginaire, il y a au plus deux choix possibles pour~$\poq$~;
 si~$\Lq$ est inclus dans~$K$, le seul choix possible pour~$\poq$ est~$\poq = \po$.
\item On note~$\Poq$ l'unique id\'eal de~$\Lq$ situ\'e au-dessous de~$\poq$.
\item D'apr\`es la proposition~\ref{prop:frobhorspbon}, il existe une racine du polyn\^ome~$\Pq(X)$ dont la classe modulo~$\Poq$ (qui est dans~$\Fp$) vaut~$\la(\sq)$~; on note~$\bq$ une telle racine.
\end{enumerate}
\end{notations}

\begin{definition}[Borne~$C(K , n)$]\label{def:C(K,n)}
Pour tout entier~$n$, on pose~:
$$
C(K , n) =  \left( n^{2h_K} C_2(K) + n^{h_K} \right)^{2d_K}.
$$
\end{definition}

\begin{proposition}\label{prop:egaglob}
Soient~$\q$ un id\'eal maximal de~$K$ premier \`a~$p$ et~$\gq$ un g\'en\'erateur de~$\q^{h_K}$ v\'erifiant l'in\'egalit\'e de la proposition~\ref{prop:hautgen}.
On suppose que~$p$ est strictement sup\'erieur \`a~$C ( K, N\q )$.
Alors on est dans l'un des trois cas suivants (avec les notations~\ref{not:po} et~\ref{not:poq})~:
 \begin{center}
 \renewcommand\arraystretch{1.5}
\begin{tabular}{| c | c | c | c |} \hline 
Type de r\'eduction de~$E$ en~$\q$ & $\la^2(\sq)$ & $\cN(\gq)$  & Cas\\ \hline 
\multirow{2}{*}{multiplicatif}  & $1 \Mod p$ & $\cN(\gq) = 1$ & M0 \\  \cline{2-4} 
    & $(N\q)^2 \Mod p$  & $\cN(\gq) = (N\q)^{2h_K}$ &  M1\\  \cline{2-4} \hline
  bon  & $\bq^2 \Mod \Poq$ & $\cN(\gq) = \bq^{2h_K}$ & B \\ \hline 
  \end{tabular}
  \end{center} 
\end{proposition}

\begin{proof}
La d\'emonstration est semblable \`a celle  de la proposition~2.4.1 de~\cite{[Dav08]} (ou de la proposition~2.14 de~\cite{[Cici]}). On en rappelle ici le principe.

D'apr\`es les propositions~\ref{prop:redmultq} et~\ref{prop:frobhorspbon} et en appliquant la proposition~\ref{prop:recmuglob} \`a l'\'el\'ement~$\gq$ de~$\OK$,
 on obtient qu'on est dans l'un des trois cas suivants~:
  \begin{center}
 \renewcommand\arraystretch{1.5}
 \begin{tabular}{| c | c | c | c |} \hline 
Type de r\'eduction de~$E$ en~$\q$ & $\la^2(\sq)$ & Congruence  & Cas\\ \hline 
\multirow{2}{*}{multiplicatif}  & $1 \Mod p$ & $\cN(\gq) \equiv 1 \Mod \po $ & M0 \\  \cline{2-4} 
    & $(N\q)^2 \Mod p$  & $\cN(\gq) \equiv (N\q)^{2h_K}  \Mod \po$ &  M1\\  \cline{2-4} \hline
  bon  & $\bq^2 \Mod \Poq$ & $\cN(\gq) \equiv \bq^{2h_K} \Mod \poq$  & B \\ \hline 
  \end{tabular}
  \end{center}
Ainsi, selon chaque cas,~$p$ divise l'un des trois entiers relatifs~$N_{K / \Q} \left( \cN(\gq) - 1 \right) $, $N_{K / \Q} \left( \cN(\gq) - (N\q)^{2h_K} \right) $ ou~$N_{K\Lq / \Q} \left( \cN(\gq) - \bq^{2h_K} \right) $.
Comme~$p$ est choisi strictement sup\'erieur \`a~$C ( K, N\q )$, qui par la proposition~\ref{prop:hautgen} est sup\'erieur ou \'egal aux valeurs absolues de ces trois entiers, on obtient que dans chaque cas, l'entier relatif en question est nul, ce qui donne la conclusion de la proposition. 
\end{proof}  

\begin{proposition}\label{prop:compat}
Soient~$q$ un nombre premier rationnel totalement d\'ecompos\'e dans~$K$ et~$\q$ un id\'eal premier de~$K$ au-dessus de~$q$.
On suppose que~$p$ est strictement plus grand que $C(K, q)$.
Alors on est dans l'un des quatre cas suivants (avec les notations~\ref{not:poq})~:
  \begin{center}
 \renewcommand\arraystretch{1.5}
 \noindent
 \begin{small} 
 \begin{tabular}{| c | c | c | c |} \hline
 R\'eduction de~$E$ en~$\q$ & $\la^2(\sq)$ & Famille $(\ato)_{\tau \in G}$ & Cas \\ \hline
\multirow{2}{*}{multiplicative} & $1 \mod p$  & $\forall \tau \in G, \ato = 0$ & M0 \\ \cline{2-4} 
  & $q^2 \mod p$  & $\forall \tau \in G, \ato = 1$ & M1 \\ \hline
bonne (ordinaire),  &  \multirow{2}{*}{$ N_{K/\Lq}(\q) = \bq\OLq$} & $\forall \tau \in \Gal(K/\Lq), \ato = 1$ & \multirow{2}{*}{BO} \\
corps~$\Lq$ quadratique, &  & et $\ato = 0$ sinon & \\ \cline{2-4}
inclus dans~$K$, &  \multirow{2}{*}{$ N_{K/\Lq}(\q) = \bqb\OLq$} & $\forall \tau \in \Gal(K/\Lq), \ato = 0$ & \multirow{2}{*}{BO'} \\
$p$ d\'ecompos\'e dans~$\Lq$ &  & et $\ato = 1$ sinon  & \\  \hline
   \end{tabular} 
  \end{small}
  \end{center}
 \end{proposition}

 \begin{proof}  
 La démonstration est semblable à celle de la proposition~2.15 de~\cite{[Cici]} (voir aussi la proposition~2.4.2 de~\cite{[Dav08]} et le lemme~2 de~\cite{[Mom]}).
  On en rappelle ici les principaux arguments, car la proposition~\ref{prop:compat} est centrale pour la d\'emonstration du th\'eor\`eme~II.
 
 Comme le nombre premier~$q$ est suppos\'e totalement d\'ecompos\'e dans l'extension galoisienne~$K/\Q$, les id\'eaux~$\tau(\q)$ sont deux \`a deux distincts lorsque~$\tau$ d\'ecrit~$G$, leur produit est l'id\'eal~$q\OK$ 
et la norme de~$\q$ dans~$K / \Q$ est \'egale \`a~$q$.
En particulier, en supposant que~$p$ est strictement sup\'erieur \`a~$C(K, q)$, on se place dans les hypoth\`eses de la proposition~\ref{prop:egaglob}.
On raisonne donc selon les trois cas de cette proposition.
 
 L'idéal engendré par~$\cN(\gq)$ dans~$\OK$ est~:
 $$
  \cN(\gq) \OK =  \left(\prod\limits_{\tau \in G} \tau(\gq)^{2\ato}\right) \OK =  \prod\limits_{\tau \in G} \left( \tau \left( \q^{h_K} \right) \right)^{2\ato}= \left( \prod\limits_{\tau \in G} \tau \left(  \q \right)^{\ato} \right)^{2h_K}
 $$

Dans les cas~M0 et~M1, la comparaison de l'id\'eal engendré dans~$\OK$ par~$\cN(\gq)$ avec celui engendré par~$1$ ou~$q^{2h_K}$ donne le résultat de la proposition.

Dans le cas~B, l'élément~$\cN(\gq)$ de~$K$ est contenu dans~$\Lq$~; ainsi, soit il est rationnel, soit il engendre~$\Lq$ (qui est alors contenu dans~$K$).

On suppose d'abord que~$\cN(\gq)$, qui est égal à~$\bq^{2h_K}$, est rationnel.
Alors~$\bq^{2h_K}$ est à la fois un entier algébrique (de~$\Lq$) et rationnel~; il appartient donc à~$\Z$.
Comme sa valeur absolue complexe est~$q^{h_K}$, on a~$\bq^{2h_K}$  égal à~$q^{h_K}$ ou~$-q^{h_K}$.
Dans les deux cas,~$\bq^{2h_K}$, qui est égal à~$\cN(\gq)$, engendre dans~$\OK$ l'idéal~:
$$
\left( \prod\limits_{\tau \in G} \tau \left(  \q \right)^{\ato} \right)^{2h_K} = \cN(\gq) \OK = \bq^{2h_K} \OK = q^{h_K} \OK = \left( \prod\limits_{\tau \in G} \tau \left(  \q \right) \right)^{h_K}.
$$ 
Ceci est impossible (avec les idéaux~$(\tau(\q))_{\tau \in G}$ deux à deux distincts et les entiers~$(\ato)_{\tau \in G}$ valant~$0$ ou~$1$).

Ainsi, l'élément~$\cN(\gq)$ de~$K$ engendre~$\Lq$, qui est quadratique imaginaire et contenu dans~$K$.
La fin de la démonstration de la proposition~2.15 de~\cite{[Cici]} (à partir de la page~20~; voir aussi la proposition~2.4.2 de~\cite{[Dav08]}) s'applique alors directement,
en prenant garde que les entiers~$(\ato)_{\tau \in G}$ de~\textit{loc. cit.} valent~$12$ fois ceux du présent texte.

On résume ici les étapes de cette démonstration.
Le fait que~$\cN(\gq)$ appartienne à~$\Lq$ donne que les entiers~$(\ato)_{\tau \in G}$ sont constants sur les classes modulo~$\Gal(K/\Lq)$, sous-groupe d'indice~$2$ de~$\Gal(K/\Q)$.
Alors~$\cN(\gq)$ engendre dans~$\OLq$ l'idéal (avec $\gamma$ un élément de~$G$ induisant sur~$\Lq$ la conjugaison complexe)~:
$$
 \cN(\gq) \OLq = \left(\bq \OLq \right)^{2h_K} =  \left(\left(N_{K/\Lq}\left(\q\right)\right)^{a_{id}} \left(\overline{N_{K/\Lq}\left(\q\right)}\right)^{a_\gamma}\right)^{2h_K}.
$$
Comme on a suppos\'e~$q$ totalement  d\'ecompos\'e dans l'extension~$K/\Q$,~$q$ est \'egalement totalement d\'ecompos\'e dans les extensions~$\Lq/\Q$ et~$K/\Lq$.
   La norme~$N_{K/\Lq}(\q)$ est donc un id\'eal premier de~$\Lq$ au-dessus de~$q$.
   Or l'\'egalit\'e~$q = \bq\bqb$ indique que les deux id\'eaux premiers (distincts) de~$\Lq$ au-dessus de~$q$ sont~$\bq\OLq$ et~$\bqb\OLq$. 
  Ainsi~$N_{K/\Lq}(\q)$  est \'egal soit \`a~$\bq\OLq$ soit \`a~$\bqb\OLq$. L'égalité  entre idéaux ci-dessus donne alors les cas BO et BO' de la proposition.
  
  Comme le polynôme~$X^2 - \Tq X + q$ (voir les notations~\ref{not:Lq}) est scindé modulo~$p$ (voir la proposition~\ref{prop:frobhorspbon}),
  l'entier~$\Tq^2 - 4q$ est un carré modulo~$p$.
  Ainsi, soit~$p$ est d\'ecompos\'e dans~$\Lq$, soit~$p$ divise l'entier relatif non nul~$\Tq^2 - 4 q$.
 Comme la valeur absolue de~$\Tq^2 - 4 q$ est~$4 q - \Tq^2 $ et inf\'erieure ou \'egale \`a~$4q$,
  le choix  de~$p$ strictement sup\'erieur \`a~$C(K , q)$ (cette borne étant strictement sup\'erieure \`a~$4q$), implique que~$p$ est d\'ecompos\'e dans le corps quadratique~$\Lq$.
  
  Enfin, on suppose par l'absurde que~$E$ a r\'eduction supersinguli\`ere en~$\q$.
  Ceci équivaut à ce que~$q$ divise~$\Tq$, entier relatif de valeur absolue inférieure ou égale à~$2\sqrt{q}$.
  Alors  on est dans l'un des cas suivants~:~$\Tq$ est nul~;~$q$ est \'egal \`a~$2$ et~$\Tq$ est \'egal \`a~$0$,~$2$ ou~$-2$~;~$q$ est \'egal \`a~$3$ et~$\Tq$ est \'egal \`a~$0$,~$3$ ou~$-3$. 
  Si~$\Tq$ est nul ou~$q$ est \'egal \`a~$3$ et~$\Tq$ vaut~$\pm 3$, alors~$\Lq$ est \'egal \`a~$\Q(\sqrt{-q})$, dans lequel~$q$ est ramifi\'e~;
  si~$q$ est \'egal \`a~$2$ et~$\Tq$ vaut~$\pm 2$, alors~$\Lq$ est \'egal \`a~$\Q( i )$, dans lequel~$2$ est ramifi\'e.
  Or, on a montr\'e que~$q$ est d\'ecompos\'e dans le corps~$\Lq$, qui est inclus dans~$K$. On en d\'eduit que~$E$ a r\'eduction ordinaire en~$\q$.
 \end{proof}

\begin{definition}[Borne $C_K$]\label{def:CK}
On pose
$$
C_K  = \max \left(  C\left( K , 2(\DK)^{Ah_K}  \right) ,   \left(  1 + 3^{\frac{d_K h_K}{2}}  \right)^2   \right).
$$
\end{definition}

Le nombre~$C_K$ ne d\'epend que du corps de nombres~$K$.

Les bornes uniformes sur l'ordre des points de torsion d'une courbe elliptique figurant dans~\cite{[Mer96]} ou~\cite{[Pa]} permettent d'\'eliminer les cas~M0 et~M1 de la proposition~\ref{prop:compat}
(voir aussi l'appendice~B de~\cite{[Krau]}).
\begin{proposition}\label{prop:mmredJK}
On suppose que~$p$ est strictement sup\'erieur \`a~$C_K$.
Alors, en tout id\'eal de~$\J_K$, la courbe~$E$ a bonne r\'eduction et tous les id\'eaux de~$\J_K$ appartiennent au m\^eme cas (BO ou BO') de la proposition~\ref{prop:compat}.
\end{proposition}

\begin{proof}
Par construction, la borne~$C_K$ est sup\'erieure ou \'egale \`a la borne~$C(K,q)$ pour tout nombre premier rationnel~$q$ totalement d\'ecompos\'e dans~$K$ et inf\'erieur ou \'egal \`a~$2(\DK)^{Ah_K}$.
Ainsi, pour tout id\'eal~$\q$ dans~$\J_K$, on est dans les conditions d'application de la proposition~\ref{prop:compat}.

Chaque cas de la proposition~\ref{prop:compat} (M0, M1, BO et BO') d\'etermine une famille diff\'erente de coefficients~$(\ato)_{\tau \in G}$ 
(on distingue les cas BO et BO'  par le fait que dans le cas BO, l'ensemble des \'el\'ements~$\tau$ de~$G$ pour lesquels~$\ato$ est \'egal \`a~$1$ est un sous-groupe de~$G$ alors que dans le cas BO', c'est le compl\'ementaire d'un sous-groupe).

Pour \'eliminer les cas M0 et M1, on raisonne comme dans le lemme~3.8 de~\cite{[Hudig]} (voir aussi la partie~3.2.1 de~\cite{[Dav08]} ou~\cite{[Cici]}, proposition~3.4).

Dans le cas M0, pour tout id\'eal premier~$\p$ de~$K$ au-dessus de~$p$, le coefficient~$\ap$ est nul.
Par d\'efinition de~$\ap$ (notation~\ref{not:ap}), cela implique que caract\`ere~$\la$ est  non ramifi\'e en toute place de~$K$ au-dessus de~$p$. 
Ainsi, l'extension~$\Kl$ de~$K$ trivialisant~$\la$, qui est ab\'elienne, est non ramifi\'ee en toute place finie de~$K$ (voir proposition~\ref{prop:inerhorsp})~;
 le corps~$\Kl$ est donc inclus dans le corps de classes de Hilbert de~$K$ et son degr\'e (sur~$\Q$) est inf\'erieur ou \'egal \`a~$d_K h_K$.
Or, la courbe~$E$ poss\`ede un point d'ordre~$p$ d\'efini sur~$\Kl$.
D'apr\`es des travaux de Merel et Oesterl\'e mentionn\'es dans les introductions de~\cite{[Pa]} et~\cite{[Mer96]}, on a alors
$$
p \leq \left(   1 + 3^{ \frac{d_K h_K}{2} }   \right)^2,
$$
ce qui contredit le choix de~$p$ strictement sup\'erieur \`a~$C_K$.

Dans le cas M1, pour tout id\'eal premier~$\p$ de~$K$ au-dessus de~$p$,~$\ap$ est \'egal \`a~$1$.
Ainsi, le caract\`ere~$\kip\la^{-1}$ est non ramifi\'e en toute place de~$K$ au-dessus de~$p$, et par suite en toute place finie de~$K$.
On consid\`ere alors le quotient de~$E$ par le sous-groupe d'isog\'enie~$W$.
On obtient une courbe elliptique~$E'$ d\'efinie sur~$K$, isog\`ene \`a~$E$~sur~$K$~; le sous-groupe~$E[p]/W$ de~$E'$ est d\'efini sur~$K$ et d'ordre~$p$ et le caract\`ere d'isog\'enie associ\'e \`a ce sous-groupe est~$\kip\la^{-1}$.
Le caract\`ere~$\kip\la^{-1}$ \'etant non ramifi\'e en toute place finie de~$K$, 
on applique le m\^eme raisonnement que pr\'ec\'edemment.
\end{proof}

\subsection{Vers un Gr\"ossencharakter}\label{ssec:sysco}

\begin{proposition}\label{prop:ordinaire}
On suppose que~$p$ est strictement sup\'erieur \`a~$C_K$.
\begin{enumerate}
\item Il existe un unique corps quadratique imaginaire~$L$ contenu dans~$K$ tel que pour tout~$\q$ dans~$\J_K$,~$\Lq$ est \'egal \`a~$L$.
\item Le corps de classes de Hilbert de~$L$ est contenu dans~$K$.  
\item Le nombre premier~$p$ est d\'ecompos\'e dans~$L$.
\item Il existe un unique id\'eal premier~$\pL$ de~$L$ au-dessus de~$p$ tel que~$\la$ est non ramifi\'e hors de~$\pL$.
\item Soit~$\p$ un id\'eal premier de~$K$ au-dessus de~$p$~; alors~$\ap$ est \'egal \`a~$1$ si et seulement si~$\p$ est au-dessus de~$\pL$~; sinon,~$\ap$ est \'egal \`a~$0$.
\item Soit~$\q$ dans~$\J_K$~; il existe une racine~$\bq'$ de~$\Pq(X)$ v\'erifiant~$N_{K/L}(\q) = \bq' \OL$ et $\la(\sq) = \bq' \mod \pL $.  
\end{enumerate}
\end{proposition}

\begin{proof}
On est dans les conditions de la partie~4.2.2 de~\cite{[Cici]} (voir aussi la partie~3.2.3 de~\cite{[Dav08]} et le théorème~1 de~\cite{[Mom]}).
Les cinq premiers points de la proposition d\'ecoulent ainsi des propositions~4.6,~4.7 et~4.8 de~\cite{[Cici]}
(ou de mani\`ere \'equivalente de la proposition~3.2.5 et des lemmes~3.2.8 et~3.2.9 de~\cite{[Dav08]}).
Pour le point 5, on notera 
que l'entier~$\ap$ de~\cite{[Cici]} et~\cite{[Dav08]} vaut~$12$ fois l'entier~$\ap$ du pr\'esent texte.

On d\'emontre enfin le point~6.
En suivant  les notations~\ref{not:poq}, on a~$ \la(\sq) = \bq \mod \Poq$.
D'apr\`es la remarque~3.2.10 de~\cite{[Dav08]} (voir aussi les d\'emonstrations des propositions~4.8 et~4.10 de~\cite{[Cici]}),
 l'id\'eal~$\pL$ du point 4 est \'egal \`a~$\Poq$ si tous les id\'eaux de~$\J_K$ sont de type~BO dans la proposition~\ref{prop:compat}
 et il est \'egal au conjugu\'e complexe de~$\Poq$ si tous les id\'eaux de~$\J_K$ sont de type~BO'.
 Or, d'apr\`es la proposition~\ref{prop:compat}, on a~$ N_{K/L}(\q) = \bq\OL$ dans le cas~BO
 et~$ N_{K/L}(\q) = \bqb\OL$ dans le cas~BO'.
 On peut donc prendre~$\bq' = \bq$ dans le cas~BO et~$\bq' = \bqb$ dans le cas~BO'.
\end{proof}

\begin{lemme}\label{lem:recfin}
Soit~$\al$ un \'el\'ement non nul de~$K$ premier \`a~$\pL$~; 
alors on a~:
 $$
 \prod\limits_{\q \nmid \pL} \la \left(\sq\right)^{\mathrm{val}_{\q}(\al)} = N_{K/L}\left(\al\right) \Mod\  \pL.
 $$ 
\end{lemme}

\begin{proof}
On note~$r'$ et~$(r'_{\nu})_{\nu}$ les applications associ\'ees par la th\'eorie du corps de classes au caract\`ere~$\la$, comme les applications~$r$ et $(r_{\nu})_{\nu}$ sont associ\'ees au caract\`ere~$\la^2$ au d\'ebut de la partie~\ref{cdc}.

D'apr\`es la proposition~\ref{prop:ordinaire}, le caract\`ere~$\la$ est non ramifi\'e hors de~$\pL$ et pour tout id\'eal premier~$\p$ au-dessus de~$\pL$, on a~$\ap = 1$.
De plus, par le point~1 de la proposition~\ref{prop:ordinaire} le corps~$K$ contient un corps quadratique imaginaire.
Ainsi, toutes ses places archim\'ediennes sont complexes et pour toute telle place~$\nu$, l'application~$r'_{\nu}$ est triviale.
On obtient donc, comme dans la d\'emonstration de la proposition~\ref{prop:recla2}~:
$$
\begin{array}{r c l}
 1 & = &  r' \left(\left(\iota_{\nu}(\al)\right)_{\nu}\right) 
    =  \prod\limits_{\nu}  r'_{\nu}\left(\iota_{\nu}(\al)\right) 
   =  \prod\limits_{\nu | \infty} 1 \times \prod\limits_{\q \nmid \pL}  r'_{\q}\left(\iota_{\q}(\al)\right) \times \prod\limits_{\p | \pL} r'_{\p}\left(\iota_{\p}(\al)\right) \\
   & = & \prod\limits_{\q \nmid \pL} \la (\sq )^{\mathrm{val}_{\q}(\al)} \times \prod\limits_{\p | \pL}  \left( N_{K_{\p}/\Q_{\p}}\left(\iota_{\p}(\alpha)\right)^{ - \ap} \Mod  p \right) \\
   & = & \prod\limits_{\q \nmid \pL} \la (\sq )^{\mathrm{val}_{\q}(\al)} \times \left( \prod\limits_{\p | \pL}   N_{K_{\p}/\Q_{\p}}\left(\iota_{\p}(\alpha)\right) \Mod  p \right)^{ - 1} .\\
\end{array}
$$
Comme le nombre premier~$p$ est d\'ecompos\'e dans~$L$ (point 3 de la proposition~\ref{prop:ordinaire}), le compl\'et\'e~$L_{\pL}$ de~$L$ en~$\pL$ co\"incide avec~$\Qp$ et on a~:
$$
\prod\limits_{\p | \pL}   N_{K_{\p}/\Q_{\p}}\left(\iota_{\p}(\alpha)\right) \Mod p  = \prod\limits_{\p | \pL}   N_{K_{\p}/L_{\pL}}\left(\iota_{\p}(\alpha)\right)  \Mod p = N_{ K / L }(\alpha) \Mod \pL.
$$

\end{proof}

\begin{proposition}\label{prop:alq}
Soit~$\q$ un id\'eal maximal de~$K$ premier \`a~$\pL$.
Il existe un unique \'el\'ement~$\alq$ de~$\OL$ v\'erifiant~:
$$
 N_{K/L}(\q) =  \alq \OL \text{ et } \la(\sq) = \alq \Mod \pL.
$$
\end{proposition}

\begin{proof}
Soit~$\q$ un id\'eal maximal de~$K$ premier \`a~$\pL$.
D'apr\`es la proposition~\ref{prop:genclid}, il existe un id\'eal premier~$\q'$ dans~$\J_K$ et un \'el\'ement non nul~$\alpha$ de~$K$ qui v\'erifient~:~$\q = \q' \cdot \left( \alpha \OK \right)$.

Soit~$\beta'_{\q'}$ dans~$\OL$ qui v\'erifie le point~6 de la proposition~\ref{prop:ordinaire} pour l'id\'eal~$\q'$.
Alors l'id\'eal (entier)~$N_{K/L}(\q)$ de~$\OL$ est engendr\'e par~$\beta'_{\q'} \times N_{K/L}(\alpha)$.
De plus, par le lemme~\ref{lem:recfin} appliqu\'e \`a~$\alpha$, on a~:
$$
\la(\sq) \la(\sigma_{\q'})^{-1} = N_{K/L}(\alpha) \Mod \pL .
 $$
Comme on a, par choix de~$\beta'_{\q'}$,~$\la(\sigma_{\q'})  =  \beta'_{\q'} \Mod \pL$,
on obtient finalement~:
$$
\la(\sq)  =  \la(\sigma_{\q'}) \left( N_{K/L}(\alpha) \Mod \pL\right) = \beta'_{\q'} \times N_{K/L}(\alpha) \Mod \pL .
$$
Ainsi, on peut prendre pour~$\alpha_{\q}$ l'\'el\'ement~$\beta'_{\q'} \times N_{K/L}(\alpha)$ de~$\OL$. 

On suppose que deux \'el\'ements~$\alq$ et~$\alq'$ de~$\OL$ satisfont la conclusion de la proposition.
 Alors~$\alq$ et~$\alq'$  engendrent le m\^eme id\'eal de~$\OL$, donc sont associ\'es par une unit\'e~$u_\q$ de~$\OL$.
 Comme~$\alq$ et~$\alq'$  ont m\^eme image (non nulle) modulo~$\pL$, l'unit\'e~$u_\q$ se r\'eduit sur~$1$ modulo~$\pL$.
 Pour~$p$ sup\'erieur ou \'egal \`a~$5$, cela implique que~$u_\q$ est \'egal \`a~$1$, donc que~$\alq$ et~$\alq'$ sont \'egaux.
\end{proof}

\begin{lemme}\label{lem:norm}
Soit~$\al$ un \'el\'ement non nul de~$K$ premier \`a~$\pL$~;  
alors on a dans~$L$~:
 $$
 N_{K/L}\left(\al\right) = \prod\limits_{\q \nmid \pL}\alq^{\mathrm{val}_{\q}(\al)}.
 $$ 
\end{lemme}

\begin{proof}
Par la proposition~\ref{prop:alq}, les éléments~$\prod_{\q \nmid \pL} \alq^{\mathrm{val}_{\q}(\al)}$ et~$N_{K/L}\left(\al\right)$ engendrent le même idéal fractionnaires de~$L$.
Ils diffèrent donc d'une unit\'e~$u_\al$ de~$L$.
Par le lemme~\ref{lem:recfin} et la proposition~\ref{prop:alq}, on a~: 
$$
 \prod\limits_{\q \nmid \pL} \alq^{\mathrm{val}_{\q}(\al)} \Mod\  \pL \equiv N_{K/L}\left(\al\right) \Mod\  \pL.
 $$
 Ainsi, l'unit\'e~$u_\al$ se r\'eduit sur~$1$ modulo~$\pL$.
 Pour~$p$ sup\'erieur ou \'egal \`a~$5$, on obtient donc que~$u_\al$ est \'egal \`a~$1$.
\end{proof}

Dans le but d'\'etendre la relation du lemme~\ref{lem:norm} \`a tous les \'el\'ements non nuls de~$K$,
 on associe \`a tout id\'eal premier~$\p$ de~$K$ au-dessus de~$\pL$ un \'el\'ement~$\alp$ de~$L^{\times}$ de la mani\`ere suivante.
 
\begin{notation}\label{not:alp}
Soit~$\p$ un id\'eal premier de~$K$ au-dessus de~$\pL$. On fixe un \'el\'ement~$x_\p$ de~$K^{\times}$ v\'erifiant~:
\begin{enumerate}
\item[(i)] $\val_{\p}(x_\p) = 1$ ;
\item[(ii)] pour tout idéal $\p'$ de $K$ au-dessus de $\pL$ différent de $\p$, $\val_{\p'}(x_\p) = 0$.
\end{enumerate}
On d\'efinit alors~:
$$
\alp = N_{K/L}(x_\p) \prod\limits_{\q \nmid \pL} \alq^{- \val_{\q}(x_\p)}.
$$
\end{notation}

\begin{lemme}\label{lem:normpartout}
Soit~$\al$ un \'el\'ement non nul de~$K$. 
Alors on a dans~$L$~:
 $$
 N_{K/L}\left(\al\right) = \prod\limits_{\nu \nmid \infty}\al_\nu^{\mathrm{val}_{\nu}(\al)}.
 $$ 
\end{lemme}

\begin{proof}
On considère~:
$
\al' = \al \prod\limits_{\p | \pL} x_\p^{- \val_\p(\al)}
$.
Par choix des~$(x_\p)_{\p | \pL} $,~$\al'$ est premier à~$\pL$~; on a donc par le lemme \ref{lem:norm}~:
$ 
N_{K/L}\left(\al' \right) = \prod\limits_{\q \nmid \pL}\alq^{\mathrm{val}_{\q}(\al')}
$.
On en d\'eduit~:
$$
\begin{array}{r c l }
N_{K/L}\left(\al \right) & = & N_{K/L}\left(\al' \right) \prod\limits_{\p | \pL}  N_{K/L}\left( x_\p \right)^{\val_\p(\al)} \\
&  =  & \prod\limits_{\q \nmid \pL}\alq^{\mathrm{val}_{\q}(\al')} \prod\limits_{\p | \pL}  \left( \alp \prod\limits_{\q \nmid \pL} \alq^{\val_\q(x_\p)} \right)^{\val_\p(\al)} \\
  & = &  \prod\limits_{\p | \pL}  \alp^{\val_\p(\al)}   \prod\limits_{\q \nmid \pL} \alq^{\val_{\q}(\al') + \sum\limits_{\p | \pL}\val_\q(x_\p)\val_\p(\al)}. \\
\end{array}
$$
Or, pour un idéal premier $\q$ de $K$ premier à $\pL$, on a :
$$
\val_{\q}(\al') = \val_{\q}(\al) - \sum\limits_{\p | \pL}\val_\q(x_\p)\val_\p(\al).
$$
On en déduit~:
$
N_{K/L}\left(\al \right) = \prod\limits_{\nu \nmid \infty}\alq^{\mathrm{val}_{\nu}(\al)}
$.
\end{proof}

\subsection{D\'emonstration du th\'eor\`eme~II}\label{ssec:thm2}

On rappelle que l'existence d'une courbe semi-stable sur~$K$ et ayant des multplications complexes sur~$K$ est caract\'eris\'ee par l'existence de certains caract\`eres des id\`eles de~$K$, selon la proposition suivante
(voir la partie~\ref{cdc} pour les notations).
\begin{proposition}\label{prop:critCM}
Les assertions suivantes sont équivalentes.
\begin{enumerate}
\item Il existe une courbe elliptique définie sur~$K$, ayant des multiplications complexes sur~$K$ et partout bonne réduction sur~$K$.
\item Il existe un corps quadratique imaginaire~$F$ contenu dans~$K$ et un caractère continu~$\varepsilon$ de~$\Ax_K$ dans~$F^{\times}$ vérifiant~:
	\begin{enumerate}
	\item[(i)] $\varepsilon$ est trivial sur le sous-groupe~$ \prod\limits_{\nu | \infty} K^{\times}_{\nu}   \prod\limits_{\nu \nmid \infty} U_{K_{\nu}} $ de~$\Ax_K$~; 
	\item[(ii)]  $\varepsilon$ co\"incide avec la norme~$N_{K/F}$ sur les idèles principales.
	\end{enumerate}
\end{enumerate} 
\end{proposition}

\begin{proof}
On renvoie \`a~\cite{[Se]},~\S4.5 (th\'eor\`eme~5 et remarque 2) et~\S3.1.
On remarquera que la semi-stabilit\'e sur~$K$ d'une courbe elliptique \`a multiplications complexes associ\'ee au caract\`ere~$\varepsilon$ provient plus sp\'ecialement de la condition (i).
\end{proof}

\begin{theoreme}[Th\'eor\`eme~II]
On suppose que le corps~$K$ est galoisien sur~$\Q$ et qu'il existe une courbe elliptique~$E$ semi-stable sur~$K$ et un nombre premier~$p$ strictement supérieur à~$C_K$ tels que la représentation~$\phiEp$ est réductible.

Alors il existe une courbe elliptique~$E'$ définie sur~$K$, ayant des multiplications complexes sur~$K$ et partout bonne réduction sur~$K$,
 telle que que la semi-simplifi\'ee de la repr\'esentation~$\phiEp$ est isomorphe \`a la repr\'esentation~$\varphi_{E',p}$.
\end{theoreme}

\begin{proof}
On est dans les conditions d'application de la partie~\ref{ssec:sysco}.
Avec les notations de cette partie, on consid\`ere le corps quadratique imaginaire~$L$ (contenu dans~$K$) donn\'e par la proposition~\ref{prop:ordinaire}
et la famille~$(\alpha_\nu )_{\nu \nmid \infty}$ d'\'el\'ements de~$L^{\times}$ donn\'ee par la proposition~\ref{prop:alq} et la notation~\ref{not:alp}.

On définit un caractère continu~$\varepsilon$ de~$\Ax_K$ dans~$L^{\times}$ par~:
\begin{enumerate}
\item[(i)] $\varepsilon$ est trivial sur le sous-groupe~$ \prod\limits_{\nu | \infty} K^{\times}_{\nu}   \prod\limits_{\nu \nmid \infty} U_{K_{\nu}} $ ;
\item[(ii)] pour toute place finie $\nu$ de $K$, $\varepsilon$ envoie toute uniformisante de $K_\nu$ sur $\alpha_{\nu}$.
\end{enumerate}
D'apr\`es le lemme~\ref{lem:normpartout},~$\varepsilon$ co\"incide donc avec la norme~$N_{K/L}$ sur les id\`eles diagonales.

Par la proposition~\ref{prop:critCM} et la partie~4.5 (remarque~2) de~\cite{[Se]}, le caract\`ere~$\varepsilon$ est associ\'e \`a une courbe elliptique~$E'$ d\'efinie sur~$K$,
ayant des multiplications complexes par un ordre de~$L$ et partout bonne r\'eduction sur~$K$.
La repr\'esentation~$\varphi_{E',p}$ de~$G_K$ associ\'ee aux points de~$p$-torsion de~$E'$ est ab\'elienne et \`a valeurs dans~$\left( \OL / p\OL \right)^{\times}$.
Comme~$p$ est d\'ecompos\'e dans~$L$ (proposition~\ref{prop:ordinaire}),~$\varphi_{E',p}$ est la somme de deux caract\`eres de~$G_K$, chacun \`a valeurs dans~$\Fpx$, et v\'erifie, pour tout id\'eal premier~$\q$ de~$K$ premier \`a~$p$~:
$$
\varphi_{E',p}(\sq) = (\alq \Mod \pL , \alq \Mod \overline{\pL}).
$$

La repr\'esentation~$\phiEp$ est r\'eductible~; sa semi-simplifi\'ee est donc la somme des deux caract\`eres~$\la$ et~$\kip \la^{-1}$, chacun \`a valeurs dans~$\Fpx$.
D'apr\`es la proposition~\ref{prop:alq}, on a pour tout id\'eal premier~$\q$ de~$K$ premier \`a~$p$~:
$$
\begin{array}{r c l}
\phiEp(\sq) & = & (\alq \Mod \pL , (N\q) \alq^{-1} \Mod \pL) \\
& = & (\alq \Mod \pL , \overline{\alq} \Mod \pL)= (\alq \Mod \pL , \alq \Mod \overline{\pL}).
\end{array}
$$ 
Ainsi, la semi-simplifi\'ee de~$\phiEp$ et la repr\'esentation~$\varphi_{E',p}$ sont isomorphes.
\end{proof} 

Le th\'eor\`eme~I de l'introduction est une cons\'equence directe du th\'eor\`eme~II.

\subsection{Exemples de corps v\'erifiant le th\'eor\`eme~I}\label{ssec:HypH}

Dans cette partie on donne quelques exemples de corps v\'erifiant le th\'eor\`eme~I.
\begin{TheoIII}
On suppose que le corps~$K$ est galoisien sur~$\Q$ et vérifie l'une des propri\'et\'es suivantes.
\begin{enumerate}
\item Le corps~$K$ ne contient le corps de classes de Hilbert d'aucun corps quadratique imaginaire. Cette situation se produit notamment dans les cas suivants~:
	\begin{enumerate}
	\item[(1a)] $K$ est de degr\'e impair sur~$\Q$~;
	\item[(1b)] $K$ est totalement r\'eel.
	\end{enumerate}
\item Pour tout corps quadratique imaginaire~$L$ dont le corps de classes de Hilbert est inclus dans~$K$, il existe une unit\'e de~$K$ (pouvant d\'ependre de~$L$) dont la norme dans l'extension~$K / L$ est diff\'erente de~$1$.  Cette situation se produit notamment dans les cas suivants~:
	\begin{enumerate}
	\item[(2a)] le degr\'e de~$K$ sur~$\Q$ est le double d'un nombre impair~;
	\item[(2b)] $K$ est égal \`a~$FM$, avec~$F$ un corps quadratique imaginaire et~$M$ un corps totalement r\'eel, galoisien, poss\'edant une unit\'e de norme sur~$\Q$ \'egale \`a~$-1$~;
	\item[(2c)] $K$ v\'erifie la condition~$(C)$ de~\cite{[Krau]}.
	\end{enumerate}
\end{enumerate}
Alors il n'existe pas de courbe elliptique définie sur~$K$, ayant des multiplications complexes sur~$K$ et partout bonne réduction sur~$K$.
\end{TheoIII}

\begin{proof}
On montre d'abord que les propri\'et\'es~1 et~2 du th\'eor\`eme impliquent sa conclusion.

On suppose pour cela qu'il existe une courbe elliptique~$E$ définie sur~$K$, ayant des multiplications complexes sur~$K$ et semi-stable sur~$K$, c'est-\`a-dire ayant partout bonne réduction sur~$K$.
Soit~$R$ l'anneau des endomorphismes de~$E$ sur~$K$.
Alors~$R$ est un ordre dans un corps quadratique imaginaire~$L$, qui est donc contenu dans~$K$ (voir \cite{[Se]} \S4.5).
De plus l'invariant~$j$ de~$E$ (qui est contenu dans~$K$) engendre sur~$L$ un corps qui contient le corps de classes de Hilbert de~$L$ (voir~\cite{[SeCM]}).

Ainsi,~$K$ contient~$L$ et son corps de classes de Hilbert.
Ceci montre que, si~$K$ v\'erifie la propri\'et\'e~1, alors une telle courbe elliptique n'existe pas.

De plus, il est associ\'e \`a la courbe~$E$ un caract\`ere~$\varepsilon$ de~$\Ax_K$ dans~$L^{\times}$ comme dans la proposition~\ref{prop:critCM}.
Les propri\'et\'es du caract\`ere~$\varepsilon$ donnent alors que la norme de toute unit\'e de~$K$ dans l'extension~$K/L$ est \'egale \`a~$1$.
Ceci montre donc que, si~$K$ v\'erifie la  propri\'et\'e~2, alors un tel caract\`ere~$\varepsilon$ n'existe pas.

Si le corps~$K$ v\'erifie~(1a) ou~(1b), alors il v\'erifie la propri\'et\'e~1 du th\'eor\`eme.

Si le corps~$K$ v\'erifie~(2a), alors~$-1$ v\'erifie la propri\'et\'e~2 pour tout corps quadratique imaginaire~$L$ contenu dans~$K$.
On suppose que le corps~$K$ v\'erifie~(2b). Soit~$u$ une unit\'e de~$M$ de norme (sur~$\Q$) \'egale \`a~$-1$.
Alors pour tout corps quadratique imaginaire~$L$ inclus dans~$K$, les corps~$L$ et~$M$ sont lin\'eairement disjoints et on a~:~$N_{K/L}(u) = N_{M / \Q}(u) = -1$.
Ainsi,~$K$ v\'erifie la propri\'et\'e~$2$.

On d\'emontre enfin que si~$K$ v\'erifie la condition~$(C)$ de~\cite{[Krau]}, il v\'erifie la propri\'et\'e~$2$ du th\'eor\`eme.
On utilise ici les notations de la partie~1.1.2 de~\cite{[Krau]}.
Soit~$n$ \'egal \`a~$\frac{d_K}{2}$.
La condition~$(C)$ donne une unit\'e~$u$ de~$K$ telle que~$1$ n'est pas racine du polyn\^ome~$H^{(u)}_n(X)$.
Or, les racines de~$H^{(u)}_n(X)$ sont les produits de~$n$ racines (pas forc\'ement distinctes) du polyn\^ome minimal de~$u$ sur~$\Q$.
En particulier,  pour tout corps quadratique imaginaire~$L$ inclus dans~$K$, la norme~$N_{K/L}(u)$ de~$u$ dans l'extension~$K/L$ est une racine de~$H^{(u)}_n(X)$.
Cette norme est donc diff\'erente de~$1$ et l'unit\'e~$u$ satisfait  la propri\'et\'e~$2$ pour tous les corps~$L$ considérés.
\end{proof}

\begin{remarques}\ 

\begin{enumerate}
\item D'apr\`es le th\'eor\`eme~1 de~\cite{[Krau]}, tous les corps de nombres satisfaisant la condition~$(C)$ de~\cite{[Krau]} v\'erifient l'assertion~2 du th\'eor\`eme~I de l'introduction. Ils v\'erifient donc aussi l'assertion n\'ecessaire~1. Le th\'eor\`eme~III donne une d\'emonstration directe de ce fait.
\item Parmi les corps satisfaisant la condition~$(C)$ de~\cite{[Krau]} et qui ne figurent pas explicitement dans les autres cas du th\'eor\`eme~III, on notera les corps de groupe de Galois isomorphe \`a un groupe sym\'etrique ou \`a un groupe altern\'e et les corps cyclotomiques engendr\'es par des racines de l'unit\'e d'ordre premier (voir le th\'eor\`eme~2 de~\cite{[Krau]}).
\end{enumerate}

\end{remarques}

\nocite{}
\shorthandoff{:}
\bibliographystyle{alpha}
\bibliography{Sest.bib}

\begin{thebibliography}{Dav11b}

\bibitem[BG96]{[BuGy]}
Yann Bugeaud and K{\'a}lm{\'a}n Gy{\H{o}}ry.
\newblock Bounds for the solutions of unit equations.
\newblock {\em Acta Arith.}, 74(1):67--80, 1996.

\bibitem[Bil11]{[Bil]}
Nicolas Billerey.
\newblock Crit\`eres d'irr\'eductibilit\'e pour les repr\'esentations des
  courbes elliptiques.
\newblock {\em Int. J. Number Theory}, 7(4):1001--1032, 2011.

\bibitem[BPR11]{[BPR]}
Yuri Bilu, Pierre Parent, and Marusia Rebolledo.
\newblock Rational points on~${X}_0^{+}(p^r) $.
\newblock {\em preprint}, 2011.
\newblock Disponible en ligne arXiv:1104.4641v1.

\bibitem[Dav08]{[Dav08]}
Agn\`es David.
\newblock {\em Caract\`ere d'isog\'enie et borne uniforme pour les
  homoth\'eties}.
\newblock PhD thesis, IRMA, Strasbourg, 2008.

\bibitem[Dav11a]{[Hudig]}
Agn\`es David.
\newblock Borne uniforme pour les homoth\'eties dans l'image de {G}alois
  associ\'ee aux courbes elliptiques.
\newblock {\em J. Number Theory}, 131(11):2175--2191, 2011.

\bibitem[Dav11b]{[Cici]}
Agn\`es David.
\newblock Caract\`ere d'isog\'enie et crit\`eres d'irr\'eductibilit\'e.
\newblock {\em Soumis}, 2011.
\newblock Disponible en ligne arXiv:1103.3892v2.

\bibitem[Kra96]{[Krau2]}
Alain Kraus.
\newblock Courbes elliptiques semi-stables et corps quadratiques.
\newblock {\em J. Number Theory}, 60(2):245--253, 1996.

\bibitem[Kra07]{[Krau]}
Alain Kraus.
\newblock Courbes elliptiques semi-stables sur les corps de nombres.
\newblock {\em Int. J. Number Theory}, 3(4):611--633, 2007.

\bibitem[LMO79]{[LMO]}
J.~C. Lagarias, H.~L. Montgomery, and A.~M. Odlyzko.
\newblock A bound for the least prime ideal in the {C}hebotarev density
  theorem.
\newblock {\em Invent. Math.}, 54(3):271--296, 1979.

\bibitem[Maz78]{[Maz]}
B.~Mazur.
\newblock Rational isogenies of prime degree (with an appendix by {D}.
  {G}oldfeld).
\newblock {\em Invent. Math.}, 44(2):129--162, 1978.

\bibitem[Mer96]{[Mer96]}
Lo{\"{\i}}c Merel.
\newblock Bornes pour la torsion des courbes elliptiques sur les corps de
  nombres.
\newblock {\em Invent. Math.}, 124(1-3):437--449, 1996.

\bibitem[Mom95]{[Mom]}
Fumiyuki Momose.
\newblock Isogenies of prime degree over number fields.
\newblock {\em Compositio Math.}, 97(3):329--348, 1995.

\bibitem[Par99]{[Pa]}
Pierre Parent.
\newblock Bornes effectives pour la torsion des courbes elliptiques sur les
  corps de nombres.
\newblock {\em J. Reine Angew. Math.}, 506:85--116, 1999.

\bibitem[Ser67]{[SeCM]}
J.-P. Serre.
\newblock Complex multiplication.
\newblock In {\em Algebraic {N}umber {T}heory ({P}roc. {I}nstructional {C}onf.,
  {B}righton, 1965)}, pages 292--296. Thompson, Washington, D.C., 1967.

\bibitem[Ser72]{[Se]}
Jean-Pierre Serre.
\newblock Propri\'et\'es galoisiennes des points d'ordre fini des courbes
  elliptiques.
\newblock {\em Invent. Math.}, 15(4):259--331, 1972.

\bibitem[ST68]{[SeTa]}
Jean-Pierre Serre and John Tate.
\newblock Good reduction of abelian varieties.
\newblock {\em Ann. of Math. (2)}, 88:492--517, 1968.

\end{thebibliography}

\end{document}